\documentclass[12pt]{amsproc}

\usepackage{geometry}                
\usepackage{graphicx}
\usepackage{amssymb}
\usepackage{float}
\usepackage{subfigure}
\usepackage{algorithm,algorithmic}

\newtheorem{theorem}{Theorem}[section]
\newtheorem{proposition}[theorem]{Proposition}
\newtheorem{lemma}[theorem]{Lemma}
\newtheorem{corollary}[theorem]{Corollary}

\theoremstyle{definition}
\newtheorem{definition}[theorem]{Definition}
\newtheorem{example}[theorem]{Example}

\theoremstyle{remark}
\newtheorem{remark}[theorem]{Remark}

\numberwithin{equation}{section}

\newcommand{\cE}{\mathcal{E}}

\newcommand{\N}{\mathbb{N}}

\newcommand{\R}{\mathbb{R}}

\newcommand{\ga}{\gamma }
\newcommand{\Ga}{\Gamma }

\newcommand{\length}{\operatorname{Length}}

\newcommand{\Mod}{\operatorname{Mod}}
\newcommand{\capa}{\operatorname{Cap}}

\newcommand{\defeq}{\mathrel{\mathop:}=}

\newcommand{\etol}{\epsilon_{\text{tol}}}

\DeclareMathOperator{\argmin}{argmin}

\title{Modulus of families of walks on graphs}

\author[Albin]{Nathan Albin$^1$}
\author[Darabi Sahneh]{Faryad Darabi Sahneh$^2$}
\author[Goering]{Max Goering$^3$}
\author[Poggi-Corradini]{Pietro Poggi-Corradini$^1$}

\address{$^1$ Department of Mathematics, Cardwell Hall, Kansas State University,
Manhattan, KS 66506, USA}
\email{albin@math.ksu.edu}\email{pietro@math.ksu.edu}

\address{$^2$ Department of Electrical and Computer Engineering, Rathbone Hall, Kansas State University, Manhattan, KS 66506, USA}
\email{faryad@k-state.edu}

\address{$^3$ Department of Mathematics, University of Washington,
Seattle, WA 98195-4350, USA}
\email{mgoering@uw.edu}

\thanks{This project was supported in part by NSF grant n. 1201427 and n.~1515810.  Goering was also supported  by the I-Center in the Dept. of Mathematics at K-State. Darabi Sahneh was also supported by the Dept. of Electrical and Computer Engineering at K-State.}


\begin{document}
\begin{abstract}
We consider the notion of  modulus of families of walks on graphs. We show how Beurling's famous criterion for extremality, that was formulated in the continuous case, can be interpreted on graphs as an instance of the Karush-Kuhn-Tucker conditions. We then develop an algorithm to numerically compute modulus using Beurling's criterion as our guide.
\end{abstract}                     
\maketitle
\baselineskip=18pt

\section{Introduction}

The theory of modulus of curve families in the plane originally introduced by Beurling and Ahlfors to solve famous open questions in function theory, has been extended over the years to families of curves in $\R^n$ and to abstract metric spaces as well. R.~J.~ Duffin in \cite{duff} developed the related notion of extremal length on graphs, mostly in the context of planar effective resistance problems.  More recently, Oded Schramm \cite{schramm} used a notion of modulus on graphs to prove a striking uniformization theorem with squares. See also \cite{cannonetal1994}  for the relation of discrete modulus with the classical Riemann mapping theorem, and \cite{hp} for a nice introduction to modulus on graphs. 

In what follows we develop a theory of modulus for families of walks. This generalizes the notion of effective conductance, which can be  recovered by restricting to walks that connect two nodes or two sets of nodes. However, we want to emphasize that there are many families of walks that are not ``connecting''.  In particular, in applications such as in network analysis the graphs are seldom planar. So it is appropriate to consider the general case. 
Along the way we try to justify our choice of using walks instead of other notions of curves (such as connected subgraphs) and show why it seems to be a better approach from the point of view of numerical computations. After a short preliminary section, we recall the famous Beurling Criterion for extremality, see \cite[Theorem 4-4, p.~61]{ahl} for the original extremal length formulation or \cite[Theorem 3.1]{ppc} for a more recent formulation using modulus. Then we prove that families of walks can be assumed to be finite, without loss of generality. And hence modulus on graphs can be categorized as a problem of ``ordinary convex optimization'', see \cite{Rockafellar}. In particular, we show that the Beurling Criterion on graphs is an instance of the Karush-Kuhn-Tucker (or KKT) conditions, see \cite{Rockafellar}. With this in mind, we develop a numerical algorithm that always terminates and gives an approximate value of modulus, within a preset tolerance.  Our approach is guided by Beurling's criterion as in \cite{schramm}, in that our algorithm tries to build what we call a Beurling subfamily as it approximates modulus. We then perform various empirical tests that suggest that this could be a reliable tool for further investigations inspired by the wealth of results about modulus in the continuous case. 

We acknowledge the anonymous referee for several helpful suggestions that have improved the exposition.

\section{Notations and Definitions}
                 
We will restrict our study to simple finite connected graphs.
Let $G= (V,E)$ be a {\sf graph} with vertex-set $V$ and edge-set $E$.
We say that $G$ is {\sf simple} if there is at most one undirected edge between
any two distinct vertices, and it is {\sf finite} if the vertex set has
cardinality $|V|\defeq N\in \N$. In this case, the edge-set $E$ can be
thought of as a subset of ${V\choose 2}$, the set of all unordered
pairs from $V$. Therefore the cardinality of $E$ is $|E|\defeq M \leq
{N\choose 2}= (N/2) (N-1)$.

We say that two vertices $x,y$ are {\sf neighbors} and write $x\sim y$
if and only if $\{x,y \}\in E$.  
\begin{definition}\label{def:walk}
A string   $W\defeq x_{0}\ e_1\ x_{1}\ e_2\ x_2\cdots e_n\ x_{n}$ with $x_i\in V$ for $i=0,\dots,n$ and $e_k=\{x_{k-1},x_k\}\in E$, for $k=1,\dots,n$,  is called a {\sf walk (with $n$ hops)} from $x_0$ to $x_n$. For simplicity, we will sometime just list the vertices visited by the walk, and write $W=(x_1,\dots,x_n)$. Also a walk that makes no hops will be called a {\sf trivial} or {\sf constant} walk.
\end{definition}
The graph $G$ is {\sf connected} if for any two vertices
$a,b\in V$ there is a walk from $a$ to $b$. It is known that
connected graphs must satisfy $|E|\geq N-1$ (induction). A walk that does not revisit any vertex  is called a {\sf (simple) path}.

The number of edges that are incident at a vertex $x$ is the
{\sf (local) degree} of $x$ and we indicate it as  $\deg(x)$. Every edge is incident
to two distinct vertices and hence contributes to two local degrees. Therefore
\[
\sum_{x\in V}\deg (x)=2|E|.
\]
This identity is sometimes referred to as the ``Handshake Lemma'', see \cite[Theorem 6.1]{aigner}.
It says that instead of counting edges, one can add degrees,
i.e., switch to $\deg(x)$ which is a function defined on $V$.

Given a subset of vertices $V^{\prime}\subset V$, we let $E
(V^{\prime})\subset E$ be all the edges of $G$ that connect pairs of
vertices in $V^{\prime}$. With this notation $G (V^{\prime})=
(V^{\prime}, E (V^{\prime}))$ is a simple graph which we call the 
{\sf subgraph induced by $V^{\prime}$}. More generally,    
a {\sf subgraph} of $G$ is a graph $G^{\prime}= (V^{\prime},
E^{\prime})$ such that $V^{\prime}\subset V$ and $E^{\prime }\subset E
(V^{\prime})$.

\begin{definition}
A {\sf curve} $\gamma=(V(\gamma), E(\gamma))$  is a connected subgraph of $G$, such that $V(\gamma) \subset V,\  E(\gamma) \subset E$.
\end{definition}
Also define the {\sf trace} of a walk $W$  to be the curve $H=(\{x_i\},\{e_j\})$ consisting of the vertices and edges traversed by $W$. Conversely, any curve $H$ is the trace of some walk.  In either case, $|E(H)|$ is smaller than the number of hops $W$ takes.

Notice also that there are finitely many curves but infinitely many walks. So at first blush one would think that families of curves are a more reasonable object to study than families of walks. Indeed the papers mentioned above \cite{schramm} and \cite{hp} do take the approach of curve families. However, as we will see, from the point of view of numerical computations, families of walks are better suited.

\section{Definition of modulus of families of walks}

Given a function $\rho : E\rightarrow\R$, define
the $\rho$-{\sf length} of a walk $W$ as in Definition \ref{def:walk} to be
 $$\ell_{\rho}(W)\defeq \rho(e_1)+\rho(e_2)+\cdots+\rho(e_n).$$

If $\rho:E\rightarrow [0,\infty)$, we say $\rho$ is a {\sf density}, and  in this case  $\rho(e)$ can be thought of  as a cost or penalization that the walk must incur  when traversing an edge $e$. Alternatively, one could  define densities on the vertex-set $V$ and  that would give rise to {\sf vertex-modulus} as opposed to {\sf edge-modulus}.
The {\sf energy} of a density $\rho$ is $\cE (\rho) \defeq \sum_{e \in E} \rho(e)^2$. More generally,
 $\cE_{p}(\rho) \defeq \sum_{e \in E} \rho(e)^p$ for $p\geq 1$.

\begin{definition}\label{def:admissible}
Given a family of walks $\Gamma$, we say that a density $\rho$ is {\sf admissible} for $\Gamma$ if $\ell_{\rho}(\gamma) \geq 1$,    
for every walk $\gamma \in \Gamma$; and we let $A(\Ga)$ be the set of all admissible densities for $\Ga$.
The {\sf modulus} of a family of walks $\Gamma$ is $$\Mod(\Gamma) \defeq \inf_{\rho \in A(\Gamma )} \cE(\rho).$$ 
A density $\rho_0\in A(\Ga)$ is called {\sf extremal} if $\cE(\rho_0) = \Mod(\Gamma)$.

More generally, for $p\geq 1$,  the {\sf $p$-modulus} of $\Ga$ is defined as $\Mod_p(\Ga)=\inf_{\rho\in A(\Ga)}\cE_p(\rho)$.
\end{definition}

Even though the given definition does not make apparent the role played by shortest paths, the following alternative definition shows that one can minimize over all densities provided the energy functional is normalized by the shortest $\rho$-length.

\begin{proposition}[Alternative Definition]
Let $\Ga$ be a non-empty family of non-trivial walks. Given a density  $\rho:E\rightarrow [0,\infty)$, define $L_{\rho}(\Gamma) \defeq \inf_{\gamma \in \Gamma} \ell_{\rho}(\gamma)$, and let $S(\Ga)\defeq\{\rho:\ L_\rho(\Ga)>0\}$. Then $$\Mod(\Gamma) = \inf_{\rho\in S(\Ga)} \frac{\cE(\rho)}{L_{\rho}(\Ga)^2}.$$
\end{proposition}
\begin{proof} 
Notice that for $\rho\in A(\Ga)$ we have $L_{\rho}(\Ga) \ge 1$.
Therefore  $$\inf_{\rho\in S(\Ga)} \frac{\cE(\rho)}{L_{\rho}(\Ga)^{2}} \le \inf_{\rho\in A(\Ga)} \frac{\cE(\rho)}{L_{\rho}(\Ga)^2}\leq \inf_{\rho\in A(\Ga)} \cE(\rho)=\Mod(\Ga).$$

Conversely suppose $\tilde{\rho}\in S(\Ga)$ is arbitrary.  Let $\rho \defeq \tilde{\rho}/L_{\tilde{\rho}}(\Gamma)$, so that $L_{\rho}(\Ga) = 1$. Then  $\rho$ is admissible and 
\[
\Mod(\Gamma) \le \cE(\rho) = \frac{1}{L_{\tilde{\rho}}(\Gamma)^2} \sum_{e \in E} \tilde{\rho}(e)^{2} = \frac {\cE(\tilde{\rho})}{L_{\tilde{\rho}}(\Gamma)^{2}}.
\]
Now minimize over $\tilde{\rho}\in S(\Ga)$.
\end{proof}

\begin{figure}[h]
\includegraphics[scale=.5]{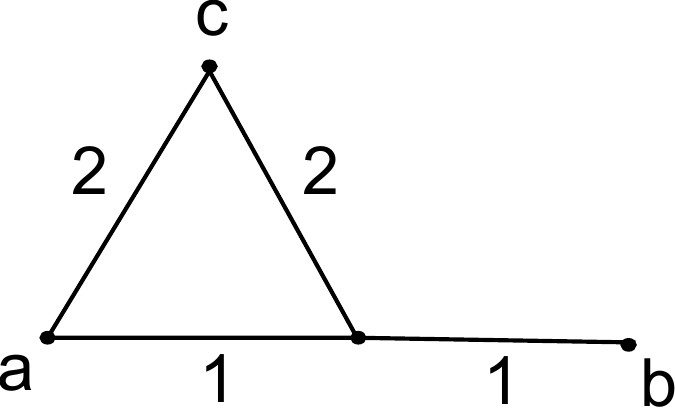}
\caption{Shortest walk or shortest curve from $a$ to $b$ through $c$}\label{fig:shortest}
\end{figure}

\begin{remark}
What is the shortest walk from $a$ to $b$ through $c$ in Figure \ref{fig:shortest}? By taking the shortest walk from $a$ to $c$ first, then the shortest walk from $c$ to $b$, one can see fairly easily that the shortest walk
has $\rho$-length equal to $5$. On the other hand, the shortest curve, i.e., the minimal connected subgraph containing $a,b$ and $c$ can be seen to have length equal to $4$.  It turns out that finding a shortest walk is a fairly easy problem to handle using  Dijkstra's algorithm, which runs in polynomial time $O(N^2)$. While finding a shortest curve through $3$ vertices is an instance of the {\it Graphical Steiner Minimal Tree Problem}, which is NP-complete. This is the main reason why, in this work,  we prefer families of walks to families of curves.
\end{remark}

\begin{proposition}[Basic Properties of Modulus]\label{prop:basic}
Let $\Ga$ (or $\Gamma_i$) be a family of walks in a finite graph $G$. The following properties hold:
\begin{itemize}
\item {\bf Constant Walks:} If $\Ga$ contains a constant walk, then $\Mod(\Ga)=\infty$.
\item {\bf Empty Family:} If $\Ga=\emptyset$,  the empty family, then $\Mod(\Ga) = 0$.
\item {\bf Monotonicity:}  If $\Gamma_{1} \subset \Gamma_{2}$, then $\Mod(\Gamma_{1}) \le \Mod(\Gamma_{2})$.
\item {\bf Countable Subadditivity:}  $\Mod(\cup_{i=1}^\infty \Gamma_i) \le \sum_{i=1}^{\infty} \Mod(\Gamma_{i})$.
\end{itemize}
\end{proposition}
\begin{proof}
If $\Ga$ contains a constant walk, then no density can be admissible, therefore $A(\Ga)=\emptyset$ and $\Mod(\Ga)=\infty$. On the other hand, if $\Ga=\emptyset$ then every density is admissible, so $\Mod(\Ga)=0$. Now, assume $\Ga_1\subset\Ga_2$. Then $A(\Gamma_2)\subset A(\Ga_1)$, so  $\Mod(\Gamma_{1}) \le \Mod(\Gamma_{2})$.
Finally, assume $\Ga_i$ are non-empty families of non-constant walks. For every
$\Gamma_{i}$ pick a density $\rho_{i}\in A(\Ga_i)$ such that $\cE(\rho_{i})\leq \Mod(\Gamma_{i}) + 2^{-i}\epsilon$. Now let $\rho\defeq (\sum_{i=1}^{\infty} \rho_i^2)^{1/2}$. Then $\rho$ is admissible for $\Gamma \defeq \cup_{i = 1}^{\infty} \Gamma_{i}$ because if $\gamma \in \Ga$, then $\gamma \in \Gamma_{i}$ for some $i$ and $\rho \ge \rho_{i}$. Therefore, 
\begin{align*}
\Mod\left(\cup_{i=1}^\infty\Ga_i\right)&\leq \cE(\rho) =\sum_{e\in E}\rho(e)^2= \sum_{e \in E} \sum_{i=1}^{\infty} \rho_{i}(e)^2 =   \sum_{i=1}^{\infty}  \sum_{e \in E} \rho_{i}(e)^2\\ &=  \sum_{i=1}^{\infty} \cE(\rho_{i}) \le \sum_{i=1}^{\infty} \left(\Mod(\Gamma_{i}) + \frac{\epsilon}{2^{i}}\right) = \sum_{i=1}^{\infty} \Mod(\Gamma_{i}) + \epsilon.
\end{align*}
 Since $\epsilon$ is arbitrary, we conclude that $\Mod(\cup_{i} \Gamma_{i=1}^{\infty}) \le \sum_{i=1}^{\infty} \Mod(\Gamma_{i})$.
\end{proof}

\begin{example}[House graph]\label{ex:house}
Here we introduce a simple example of a graph and a family of walks connecting two nodes that we will examine more thoroughly in Section \ref{sec:becrit}.

\begin{figure}[H]
\includegraphics[scale=0.25]{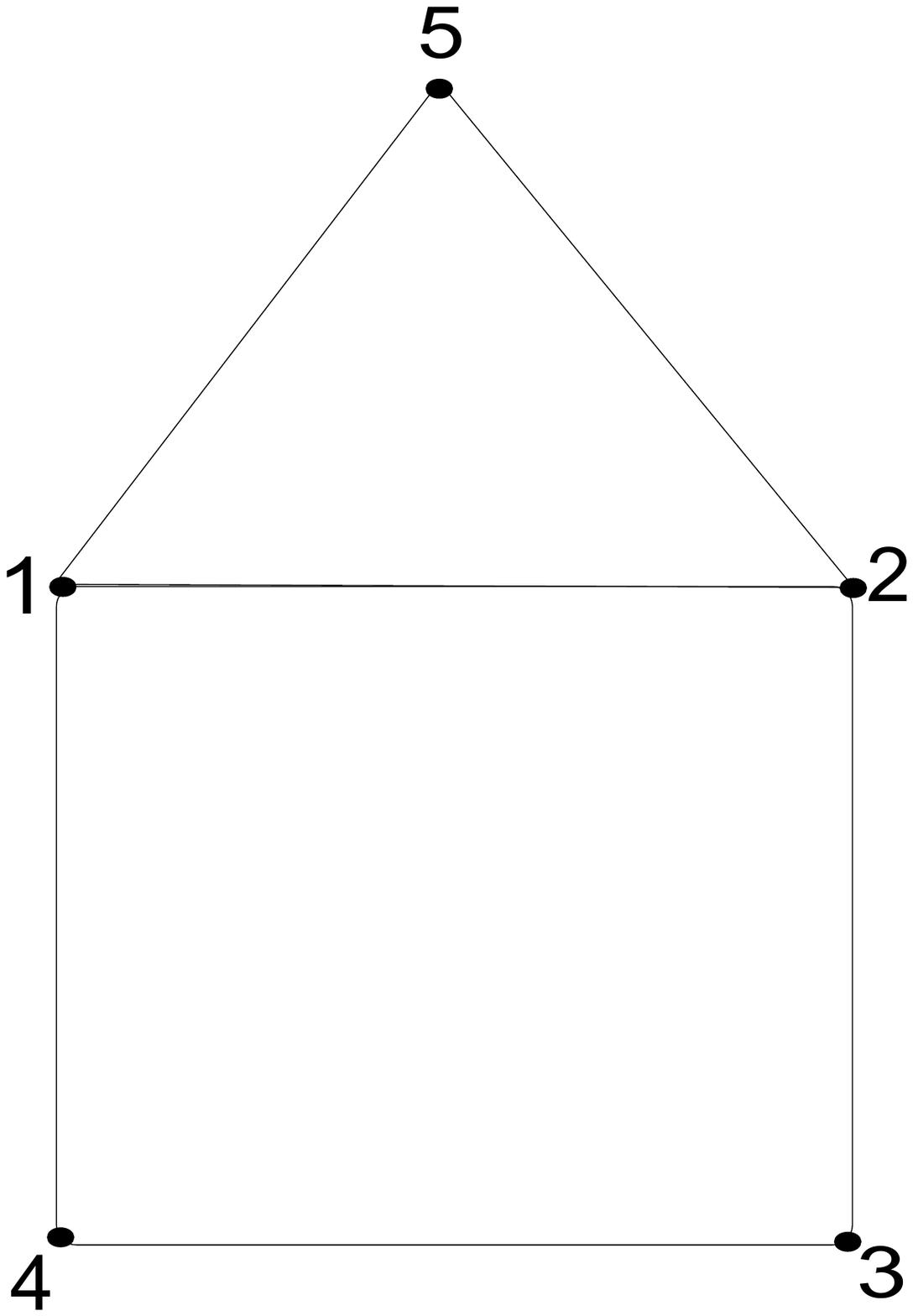}
\caption{House graph}\label{fig:house}
\end{figure}

For the graph in Figure \ref{fig:house}, we may define the family of walks 
\[
\Ga=\Gamma(1,2) \defeq \{\mbox{all walks from $1$ to $2$}\}.
\]   
In order to compute the modulus of this family we need to introduce some facts about extremal densities first.
\end{example}

\section{Beurling's criterion for extremality}\label{sec:becrit}
 
\begin{theorem}[Existence and uniqueness of extremal densities]\label{thm:existuniq}
Given a family of non-constant  walks $\Gamma$ in a finite graph $G$, consider its modulus $\Mod(\Gamma)$. Then
an extremal metric $\rho_0\in A(\Ga)$ exists and is unique.
\end{theorem}

We defer the proof of Theorem \ref{thm:existuniq} to Section \ref{sec:kkt}.

\begin{theorem}[Beurling's Extremality Criterion]
\label{thm:beurling}
A density $\rho\in A(\Ga)$ is extremal, if there is $\tilde{\Ga}\subset\Ga$ with $\ell_\rho(\ga)=1$ for all $\ga\in \tilde{\Ga}$ such that: 
\begin{equation}
\label{eq:beurling}
\mbox{$\sum_{e\in E}h(e)\rho(e)\geq 0$ whenever $h:E\rightarrow\R$ with $\ell_h(\ga)\geq 0$ for all $\ga\in\tilde{\Ga}$.}
\end{equation} 
\end{theorem}

The proof is very simple and well-known, so we reproduce it here for completeness.

\begin{proof}[Proof of Theorem \ref{thm:beurling}]
Let $\sigma$ be an arbitrary admissible metric for $\Ga$. Let $h \defeq \sigma - \rho$. If $\gamma \in \tilde{\Gamma}$, then 
\[
\ell_h(\ga)=\ell_\sigma(\ga)-\ell_\rho(\ga)\geq 1-1=0.
\]
By assumption (\ref{eq:beurling}), this implies  $\sum_{e\in E} h(e)\rho (e) \ge 0$. Hence, $\sum \sigma(e)\rho(e) - \sum \rho(e)^2 \ge 0$. So by the Cauchy-Schwarz inequality:
\[
\cE(\rho) =\sum_{e\in E} \rho(e)^2 \le \sum_{e\in E} \sigma(e)\rho(e) \le \sqrt{\sum_{e\in E} \sigma(e)^2} \sqrt{\sum_{e\in E}\rho(e)^2}.
\]
Dividing both sides $\sqrt{\cE(\rho)}$ and squaring we get that $\cE(\rho) \le \cE(\sigma)$.
\end{proof}

\begin{definition}
If $\tilde{\Ga}\subset\Ga$ is a subfamily as in Beurling's Criterion (Theorem \ref{thm:beurling}), we say that $\tilde{\Ga}$ is a {\sf Beurling subfamily}.
\end{definition}

\begin{theorem}[Converse of Beurling's Criterion]\label{thm:conversebeurling}
 If  $\rho\in A(\Ga)$ is extremal, then $$\Ga_0(\rho)\defeq\{\ga\in\Ga: \ell_\rho(\ga)=1\}$$ is a Beurling subfamily.
\end{theorem}
We defer the proof of Theorem  \ref{thm:conversebeurling} to Section \ref{sec:kkt}.

\begin{corollary}[Beurling Subfamilies]\label{cor:subfamily}
If $\tilde{\Ga}\subset\Ga$ is a Beurling subfamily, then $$\Mod(\Ga)=\Mod(\tilde{\Ga}).$$
\end{corollary}

\begin{proof}[Proof of Corollary \ref{cor:subfamily}]
Since $\tilde{\Ga}$ is a Beurling subfamily, letting $\rho_0\in A(\Ga)$ be the extremal density for $\Ga$, we see that $\tilde{\Ga}\subset \Ga_0(\rho_0)$ and (\ref{eq:beurling}) holds for $\rho_0$.

By monotonicity, $\Mod(\tilde{\Ga})\leq \Mod(\Ga)$. Conversely, let  $\rho\in A(\tilde{\Ga})$. Define $h:E\rightarrow\R$ as $h=\rho-\rho_0$. Then, for every $\ga\in\tilde{\Ga}$ we have $\ell_h(\ga)=\ell_\rho(\ga)-\ell_{\rho_0}(\ga)\geq 1-1=0$. So by (\ref{eq:beurling}), we have $\sum_{e\in E} h(e)\rho_0(e)\geq 0$. This implies
$\sum_{e\in E}\rho(e)\rho_0(e)\geq \cE(\rho_0)$ and by Cauchy-Schwarz we see that $\cE(\rho)\geq \cE(\rho_0)=\Mod(\Ga)$. Minimizing over all $\rho\in A(\tilde{\Ga})$ we find that $\Mod(\tilde{\Ga})\geq \Mod(\Ga)$.
\end{proof}

Let's revisit Example \ref{ex:house}.
\begin{example}[House graph continued]\label{ex:housecont}
Define $\rho_0$ as in Figure \ref{fig:rhodens}: 
\begin{figure}[H]
\includegraphics[scale=0.25]{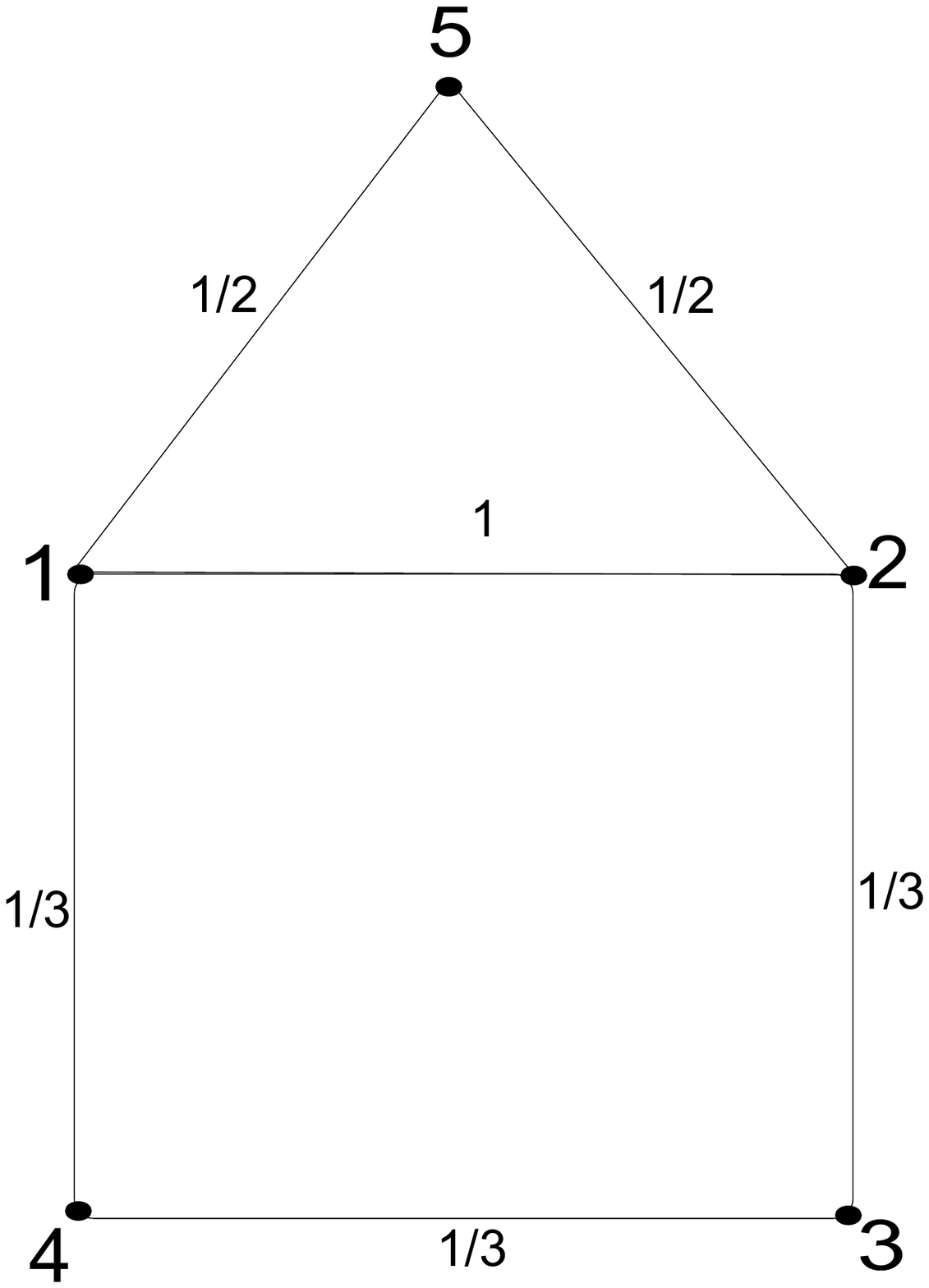}
\caption{The density $\rho_0$}\label{fig:rhodens}
\end{figure}
We can test $\rho_0$ using Beurling's Criterion. Consider the paths $P_{1} = (1,2), P_{2} = (1,5,2), P_{3} = (1,4,3,2)$, and let $\tilde{\Gamma} \defeq \{ P_1, P_2, P_3\}$. Given a function $h:E\rightarrow \R$, if $e=\{x,y\}$, we write $h(e)$ as $h(x,y)$. Suppose $h$ is such that $\ell_h(P_1)\ge 0$. That means $h(1,2) \ge 0$ and therefore  $h(1,2)\rho_{0}(1,2) \ge 0$. Similarly,  if $\ell_h(P_2) \ge 0$, then  $h(1,5) + h(5,2) \ge 0$. So $h(1,5)\rho_{0}(1,5) + h(5,2)\rho_{0}(5,2) = (1/2)\ell_h(P_2) \ge 0$. Finally, if $\ell_h(P_3) \ge 0$, then $\sum_{e\in P_3} h(e)\rho_{0}(e) = (1/3)\ell_h(P_3) \ge 0$. Therefore, since $E=E(P_1)\cup E(P_2)\cup E(P_3)$, by Beurling's Criterion, we know that $\rho_{0}$ is extremal for $\Ga$. Hence, 
\[
\Mod(\Ga)=\cE(\rho_{0}) = 1 + 1/4 + 1/4 + 1/9 +1/9+ 1/9 = 1.8\overline{3}.
\] 
\end{example}

\section{General properties for modulus of walk families}

In this section we review some standard properties of modulus.

\begin{definition}\label{def:curve-partial-ordering}
 Given two walks $\gamma_1$ and
$\gamma_2$ on a graph $G=(V,E)$, we say that
\begin{equation}
  \label{eq:curve-partial-ordering}
  \gamma_1\preceq \gamma_2 \qquad\iff\qquad
  \ell_\rho(\gamma_1) \le \ell_\rho(\gamma_2)\qquad\forall\rho:E\mapsto[0,\infty),
\end{equation}
with strict inequality $\gamma_1\prec\gamma_2$ if  the
inequality on the right side is strict for some $\rho$.
\end{definition}

\subsection{Shorter Walks:}\label{ssec:shorterwalk}{\em  Suppose $\Gamma_1$ and $\Gamma_2$ are two walk families.  If for every $\gamma \in \Gamma_2$ there exists  $\sigma \in \Gamma_1$ such that $ \sigma \preceq \gamma$ then $\Mod(\Gamma_2) \le \Mod(\Gamma_1)$.}

\begin{proof}
Recall that $\sigma\preceq \gamma$ means that $\ell_\rho(\sigma)\leq \ell_\rho(\gamma)$ for every density $\rho$.
Therefore, if $\rho \in A(\Gamma_1)$, then $\rho \in A(\Gamma_2)$. So $\min_{\rho \in A(\Gamma_1)} \cE (\rho) \ge \min_{\rho \in A(\Gamma_2)} \cE (\rho)$.
\end{proof}
\subsection{Basic Estimate} {\em Assume $\Gamma$ is a walk family such that $\length_G(\gamma) \ge L\  >  0$,  $\forall \gamma \in \Gamma$. Then, $\Mod(\Gamma) \le \frac{|E|}{L^{2}}$.}
\begin{definition}\label{def:graphlength}
Here we write $\length_G(\ga)$ to mean the number of hops that $\ga$ takes.
\end{definition}
\begin{proof}
Define $\rho \equiv \frac{1}{L}$. Then, since $\min_{\gamma \in \Gamma}\length_G(\gamma) \ge L$, we have $\rho  \in A(\Gamma)$. Thus, by definition of modulus, $\Mod(\Gamma) \le \cE (\rho) = \frac{|E|}{L^2}$.
 \end{proof} 

\subsection{Symmetry Rule} {\em Suppose $T: V \rightarrow V$ is a bijection and a graph isomorphism, that is, $(v_1, v_2)$  $\in E \iff (T(v_1),T(v_2)) \in E$, such that $T^2$ = I. Also assume $\Gamma$ is a family of walks that is $T$-invariant, namely, such that $T(\Gamma) = \Gamma$. Consider the set of $T$-invariant densities $\tilde{A}(\Gamma) \defeq \{ \rho \in A(\Gamma) | \rho = \rho \circ T \}$. Then: $$\Mod(\Gamma) = \inf_{\rho \in \tilde{A}(\Gamma)} \cE (\rho)$$ 
}
This says that in the presence of symmetry one can restrict to symmetric densities.
\begin{proof}
Let $\rho \in A(\Gamma).$ Define $\rho_1 \defeq \rho \circ T$. Then, $\ell_{\rho_{1}}(T\circ\gamma) =  \ell_{\rho}(\gamma)$ and since $T(\Gamma)=\Gamma$, we also have $\rho_1\in A(\Gamma)$. Moreover $\cE (\rho_1) = \cE (\rho)$. Now, if $\rho_2 \defeq \frac{\rho + \rho_1}{2}$ then $\rho_2 \in A(\Gamma)$ and using the fact that  $T$ is an involution, $\rho_2 = \frac {\rho_1 + \rho}{2} = (\frac {\rho + \rho_1}{2})\circ T = \rho_2 \circ T$, so $\rho_2 \in \tilde{A}(\Gamma)$. Moreover, $\cE (\rho_2) = \frac{1}{4} \cE (\rho) + \frac{1}{4} \cE (\rho_1) + \frac{1}{2} \sqrt{\cE (\rho)} \sqrt{\cE (\rho_1)} \le \frac{1}{2}(\cE (\rho) + \cE (\rho_1)).$ Thus, $\forall \rho \in \Gamma, \exists \tilde{\rho} \in \tilde{A}(\Gamma)$ such that $\cE (\tilde{\rho}) \le \cE (\rho)$.
\end{proof}

\subsection{Connecting Families}
 \begin{definition}[Connecting Families]\label{def:conncf}
 If $A$ and $B$ are subsets of the vertex-set in a graph  $H$, let $$\Gamma(A,B;H) \defeq \{\mbox{all walks in } H | \gamma: A \leadsto B \}.$$ The notation $\gamma: A \leadsto B$ means that the walk starts on $A$ and ends on $B$. Also a subset 
 $C \subset V(H)$ is a {\sf cut} for $\Gamma(A,B;H)$ if $\forall \gamma \in \Gamma(A,B;H)$, $V(\gamma) \cap C \neq \emptyset$.
 \end{definition}

\subsection{Extension Rules}{\em Let $\Gamma = \Gamma(A,B;H)$ be a connecting family of walks.  The following holds. 
\begin{itemize}
\item[{\rm (i)}] If  $B\subset B'$, then $\Mod(\Gamma(A,B;H)) \le \Mod(\Gamma(A,B';H))$.

\item[{\rm (ii)}] If $C$ is a cut for $\Gamma$, then $\Mod(\Gamma(A,B;H)) \le \Mod(\Gamma(A,C;H))$.

\item[{\rm (iii)}] If  $H \subset H'$, then $\Mod(\Gamma(A,B;H)) \le \Mod(\Gamma(A,B;H'))$.
\end{itemize} }
\begin{proof} 
{\rm (i)} Let $\Gamma'=\Gamma(A,B';H)$. If $\rho \in A(\Gamma')$, then $\rho \in A(\Gamma)$. So $\min_{\rho \in A(\Gamma')} \cE (\rho) \ge \min_{\rho \in A(\Gamma)} \cE (\rho)$.

{\rm (ii)} Note that for every $\gamma \in \Gamma$, there exists $\sigma \in \Gamma(A,C;H)$ such that $\sigma \preceq \gamma$, now apply the Shorter Walk property \ref{ssec:shorterwalk}.

{\rm (iii)} Observe that, if $\Gamma' \defeq \Gamma(A,B;H')$, then $\Gamma \subset \Gamma'$ and the claim follows by the Monotonicity property, see Proposition \ref{prop:basic}.
\end{proof}

\subsection{Parallel Rule}{\em
Suppose $\Gamma = \Gamma_{1}(A_1,B_1;H) \cup \Gamma_{2}(A_2,B_2;H)$ and no trace of a walk in $\Gamma_1$ ever crosses a trace of a walk in $\Gamma_2$, that is, $V(\gamma_1)\cap V(\gamma_2)=\emptyset$ for every $\gamma_1\in \Gamma_1$ and $\gamma_2\in \Gamma_2$. Then $$\Mod(\Gamma) = \Mod(\Gamma_1) + \Mod(\Gamma_2).$$
}
\begin{proof} 
Observe that $\Mod(\Gamma) \le \Mod(\Gamma_1) + \Mod(\Gamma_2)$ by subadditivity (Proposition \ref{prop:basic}).

On the other hand, for $i=1,2$, let $E_i$ consists of all the edges traversed by walks in $\Gamma_i$. Then $E_1\cap E_2=\emptyset$. So
given $\rho \in A(\Gamma)$, define $\rho_i=\rho|_{E_i}$. Then $\rho_i \in A(\Gamma_i)$, and $\cE(\rho)\geq \cE(\rho_1)+\cE(\rho_2)$. Taking the infimum of both sides, we get $\min_{\rho \in A(\Gamma)} \cE (\rho) \ge \min_{\rho_1 \in A(\Gamma_{1})} \cE (\rho_1) + \min_{\rho_2 \in A(\Gamma_{2})} \cE (\rho_2)$. Therefore, $\Mod(\Gamma) \ge \Mod(\Gamma_1) + \Mod(\Gamma_2)$  and we have we reach the desired conclusion.
\end{proof}

\subsection{Serial Rule}{\em  Let $C$ be a cut for $\Gamma\defeq\Gamma(A_1,A_2;H)$, such that $C\cap A_j=\emptyset$, for $j=1,2$. Define $\Gamma_j = \Gamma(A_j,C;H)$, $j=1,2$. Then,  $$\frac{1}{\Mod(\Gamma)} \ge \frac{1}{\Mod(\Gamma_1)} + \frac{1}{\Mod(\Gamma_2)}.$$
}

\begin{proof}
First, define $V_1$ the be the set of vertices $v$ for which any walk $\gamma: v\leadsto A_2$ must necessarily cross $C$. 
Likewise, $V_2$ is the set of vertices $w$ for which any walk $\gamma: w\leadsto A_1$ must necessarily cross $C$. Note that $C\subset V_1\cap V_2$. Also $V=V_1\cup V_2$, because if one could walk from a vertex $v$ to $A_1$ and $A_2$ without crossing $C$, then $C$ wouldn't be a cut.

For $i=1,2$,  let $H_i$ be the subgraph of $H$ induced by $V_i$. Given $\tilde{\rho}_i \in A(\Gamma_i)$, set $\rho_{i} = \tilde{\rho}_i|_{E(H_i)}$. Moreover set $\rho_i(e)=0$ if $e$ connects two vertices in $C$, so that $\rho_1$ and $\rho_2$ have disjoint supports. Then $\rho_i \in A(\Gamma_i)$. Indeed, given a path $\gamma: A_1\leadsto C$, let $\tilde{\gamma}$ be the same path stopped the first time it visits $C$. Then $\tilde{\gamma}$ does not visit any vertices in $V_2$ or traverse any edge connecting two vertices in $C$. So $$\ell_{\rho_1}(\ga)\geq\ell_{\rho_1}(\tilde{\ga})=\ell_{\tilde{\rho}_1}(\tilde{\ga})\geq 1.$$

For some $a,b \in \R$ define: $$ \rho = 
\begin{cases} a \rho_1, & \text{ on } E(H_1) \\
		    b \rho_2 , & \text{ on } E(H_2).
\end{cases}$$

Now observe that for any $\gamma \in \Gamma$, there exist $\gamma_i \in \Gamma_i$, with $E(\gamma_i) \subset E(H_i)$ such that $\ga_i\preceq \ga$. Therefore, $\ell_{\rho}(\gamma) \ge \ell_{\rho}(\gamma_1) + \ell_{\rho}(\gamma_2) = a \ell_{\rho_1}(\gamma_1) + b \ell_{\rho_2}(\gamma_2) \ge a + b$. The last inequality is due to the fact that $\rho_i \in A(\Gamma_i) \text{ for } i \in \{1,2\} $. Now, choose $a,b$ such that $a + b = 1$. Then, $\rho \in A(\Gamma)$. Moreover, $\cE(\rho) = \sum_{e \in E} \rho(e)^{2} = a^2 \sum_{e \in E(H_1)} \rho_1(e)^2 + b^2 \sum_{e \in E(H_2)} \rho_2(e)^2$. For convenience, let
$\alpha =\sum_{e \in E(H_1)} \rho_1(e)^2$ and $ \beta = \sum_{e \in E(H_2)} \rho_2(e)^2$. Since $b = 1 - a$, we now have $\cE(\rho) = a^2(\alpha + \beta) - 2a\beta + \beta$. We want to
minimize $\cE(\rho)$ over the free variable $a$. Completing the squares yields 
\begin{align*}
\cE(\rho) 
&  = \frac{1}{\alpha+\beta}\left[ ((\alpha+\beta)a - \beta)^2 + \alpha\beta\right]
\end{align*}
 With the goal of minimizing $\cE(\rho)$, since $((\alpha+\beta)a - \beta)^2$ is nonnegative, choose $a$ so this term becomes zero. Then $\cE(\rho) = \frac{\alpha\beta}{\alpha+\beta} = \frac{1}{\frac{1}{\alpha} + \frac{1}{\beta}}$. Since $\Mod(\Gamma) \le \cE(\rho)$, we find that $\frac{1}{\Mod(\Gamma)} \ge \frac{1}{\alpha} + \frac{1}{\beta}$. Since $\tilde{\rho}_i$  were arbitrary, letting them vary over $A(\Gamma_i)$, we reach our conclusion.
\end{proof}

\section{Capacity}

\subsection{Potentials and gradients}
Let $u:V \rightarrow \R$ be a function (or {\sf potential}). The {\sf gradient} of $u$ is a density $\rho_{u} : E \rightarrow [0,+\infty)$ defined on $e=\{x,y\}$ by the formula $$\rho_{u}(e) \defeq | u(x) - u(y) |.$$ 
We say  the pair $(A,B)$ forms a {\sf condenser} and define its {\sf capacity} as
\[
\capa(A,B) \defeq \inf_{u|_{A}=0,u|_{B}=1} \cE(\rho_u).
\]
A function $U:V \rightarrow \R$ that attains the infimum is called a {\sf capacitary function}. Such a function $U$ always exists and is unique. This is because $\cE (\rho_{_U}) = \frac{1}{2} \sum_{e \in E} \rho_{_U}(e)^2 = U^\top LU$ where $L$ is the combinatorial Laplacian. Since $L$ is symmetric, $U^\top LU$ is a quadratic form and the minimization can be handled by the method of Lagrange multipliers. Also  $U$ is harmonic, i.e. $LU=0$,  on $V \setminus (A\cup B)$. Therefore, uniqueness can be derived from the maximum principle. For more details see \cite{epcz}.

\begin{proposition}\label{prop:capamod}
We always have
$$ \capa(A,B) = \Mod(\Gamma(A,B)).$$
\end{proposition}

\begin{proof} 
Suppose $U$ is the capacitary function for $(A,B)$, and let $\gamma \in \Gamma$ be a walk from $A$ to $B$. Assume that $\gamma$ visits the vertices $\{a=x_0,...,x_n=b\}$ where $\{x_i,x_{i+1}\} \in E, \forall i=0,...n-1$, so that $a\in A$ and $b\in B$. Then 
\begin{align*}
1 & \leq |U(a) - U(b)| =\left|\sum_{i=0}^{n-1} (U(x_{i+1}) - U(x_{i}))\right| \\
&\leq \sum_{i=0}^{n-1} |U(x_{i+1}) - U(x_{i})|  = \sum_{i=0}^{n-1} \rho_{ _{U}} (\{x_{i+1},x_{i}\}) = \ell_{\rho_{_U}}(\gamma).
\end{align*}
 Therefore $\rho_U \in A(\Gamma)$, and hence $\capa(A,B)\geq \Mod(\Gamma(A,B))$.

 Conversely, let $\rho \in A(\Gamma)$. Define $U(x) = \inf_{\gamma: A \leadsto x} \ell_{\rho}(\gamma)$. Note that: $U|_{A} \equiv 0$ and $U|_{B} \ge 1$. If $e = \{x,y\}$ then (i) $U(y) \le U(x) + \rho(e)$ and (ii) $U(x) \le U(y) + \rho(e)$. Together (i) and (ii) imply $|U(x) - U(y)| \le \rho(e)$. Therefore, $\rho_{_U} \le \rho$. At this point, if $U|_{B} \equiv 1$, the proof is complete. So, assume $U|_{B} > 1$. In this case we use truncation. Define $\nu(x) = \min \{U(x),1\}$. Given $e = \{x,y\}$ we have three cases:
\begin{itemize}
\item[{\rm Case 1:}] $U(x)$ and $U(y) < 1$. Then, $\rho_{\nu}(e) = \rho_{_U}(e)$. 
\item[{\rm Case 2:}] $U(x) \ge 1$ and $U(y) < 1$ Then, $\rho_{\nu}(e) \le \rho_{_U}(e)$. 
\item[{\rm Case 3:}] $U(x)$ and $U(y) \ge 1$. Then, $\rho_{\nu}(e) = 0 \le \rho_{_U}(e)$. 
\end{itemize}
 In each case, $\rho_{\nu}(e) \le \rho_{_U}(e)$. So $\capa(A,B)\leq \cE(\rho_\nu)\leq \cE(\rho)$ and minimizing over $\rho$ yields the claim.
\end{proof}

\subsection{Effective conductance}\label{ssec:effcond}

In the special case when $A=\{a\}$ and $B=\{b\}$, we write $\capa(a,b)=\Mod(a,b)$ in  Proposition \ref{prop:capamod}, 
and this is seen to also equal the effective conductance (the reciprocal of effective resistance) between $a$ and $b$, see \cite{duff} or  \cite[Theorem 6.3]{epcz} where this connection is described in detail. Moreover, effective resistance between pairs of points can be computed exactly by diagonalizing and inverting the Laplacian matrix. This is explained and reviewed  in Section 5 of  \cite{epcz}. From this point of view, modulus of families of walks is a far reaching generalization of effective conductance when interpreting a graph as an electrical network.

\section{Essential subfamilies and inequality conditions for admissibility}

In this section, we show that for any family of walks $\Gamma$ in a finite graph, there
exists a finite subfamily $\Gamma^*\subset\Gamma$
with the property that $A(\Gamma^*)=A(\Gamma)$.
This will be important in Section~\ref{sec:kkt}.
We call such finite subfamily $\Gamma^*$ an  {\sf essential subfamily}. This can be useful in
certain applications, since it implies that the condition $\rho\in
A(\Gamma)$ can always be replaced by a finite system of linear
inequalities: $\rho\ge 0$ and $\ell_\rho(\gamma)\ge 1$ for all
$\gamma\in\Gamma^*$.

The construction of $\Gamma^*$ can be understood through a partial
ordering $\preceq$ of walks defined in (\ref{eq:curve-partial-ordering}). 
Define the nonnegative edge {\sf multiplicity} of a family of walks $\Ga$ as
\begin{equation*}
  m:\Gamma\times E\to\mathbb{N}_0\qquad m(\gamma,e) = 
  \text{ the number of times $\gamma$ crosses edge $e$ }.
\end{equation*}

Choose an enumeration of the edges and define a 
mapping from $\Gamma$ into $\mathbb{N}_0^{|E|}\subset \R^{|E|}$ by associating to each walk its vector of multiplicities:
\begin{equation}
  \label{eq:curves-to-N}
  \gamma\mapsto x_\gamma := \left(m(\gamma,e_1),m(\gamma,e_2),\cdots,m(\gamma,e_{|E|})\right)
\end{equation}

Consider the natural partial ordering on $\R^{|E|}$ given by:
\begin{equation}
  \label{eq:partial-ordering}
  x \preceq y \qquad\iff\qquad x_i\le y_i\quad i=1,2,\ldots,|E|,
\end{equation}
and the associated strict inequality if additionally $x\ne y$. 

Then the mapping (\ref{eq:curves-to-N}) is order-preserving, that is to say:
\begin{equation}\label{eq:orderpreserve}
\gamma_1\preceq \gamma_2 \qquad\iff\qquad x_{\gamma_1}\preceq x_{\gamma_2}.
\end{equation}
To check this, note that letting $\rho_i=1$ on $e_i$ and $0$ otherwise, we get that $\ell_{\rho_i}(\ga)=m(\ga,e_i)$. So
$\gamma_1\preceq \gamma_2$ implies $x_{\gamma_1}\preceq x_{\gamma_2}$. The converse follows from the fact that
\begin{equation}\label{eq:ellrhom}
\ell_\rho(\ga)=\sum_{i=1}^{|E|} m(\ga,e_i)\rho(e_i).
\end{equation}

In
this section, we prove the following theorem.

\begin{theorem}
  \label{thm:essential-subfamily}
  Any family of walks $\Gamma$ admits an essential subfamily, that is,  a finite subfamily
  $\Gamma^*\subseteq\Gamma$ such that $A(\Gamma^*)=A(\Gamma)$.
\end{theorem}

This result arises as a corollary of the following theorem on sets of
vectors of nonnegative integers.

\begin{theorem}
  \label{thm:essential-subset}
  Let $X\subseteq \mathbb{N}_0^{n}$.  Then there exists a finite
  subset $X^*\subseteq X$ such that
  \begin{equation*}
    X = \bigcup_{x^*\in X^*}\left\{x\in X: x^*\preceq x\right\}.
  \end{equation*}
\end{theorem}
The proof of Theorem~\ref{thm:essential-subset}, in turn, relies on
the following lemma.
\begin{lemma}
  \label{lem:minimal-elt}
  Every non-empty $X\subseteq\mathbb{N}_0^n$ has a minimal element (not necessarily unique)
  with respect to the partial ordering~\eqref{eq:partial-ordering}.
\end{lemma}

\begin{proof}
  Define the mapping $h:\mathbb{N}_0^n\to \mathbb{N}_0$ by
  \begin{equation*}
    h(x) = \sum_{i=1}^n x_i.
  \end{equation*}
  (When $x=x_\gamma$, $h(x)$ counts the number of ``hops'' in
  $\gamma$.)  By the well-orderedness of the natural numbers, the set
  $\{h(x) : x\in X\}$ has a smallest value.  Let $x\in X$ be such that
  $h(x)$ equals this value.  Then $x$ must be a minimal element.
  Indeed, suppose that $y\in X$ with $y\preceq x$.  Then $y_i\le x_i$
  for all $i=1,2,\ldots,n$.  But $h(x)\le h(y)$ then implies that no
  $y_i$ can be strictly less than the corresponding $x_i$ and
  therefore that $y=x$.
\end{proof}

\begin{proof}[Proof of Theorem~\ref{thm:essential-subset}]

  We proceed by induction.  In the case, $n=1$, the theorem follows
  from the fact that $\mathbb{N}_0$ is well-ordered; $X^*$ can be
  taken as the single smallest element of $X$.  Now, suppose the
  result holds for dimensions $k=1,2,\ldots,n-1$; we will show that it
  also holds for dimension $n$.  Let $X\subseteq\mathbb{N}_0^n$.  If
  $X$ is empty, we are finished.  Otherwise, let $x^0$ be a minimal
  element of $X$, whose existence is guaranteed by
  Lemma~\ref{lem:minimal-elt}.  Let $X^0$ be the set of elements in
  $X$ that do not dominate $x^0$:
  \begin{equation*}
    X^0 = \{x\in X : x^0\not\preceq x\}.
  \end{equation*}
  For any subset $I$ of $\mathcal{I}:=\{1,2,\ldots,n\}$, define
  \begin{equation*}
    X^0_I = \{x\in X^0 : x_i < x^0_i, \forall i\in I,\mbox{ and }\ x_i\geq x^0_i, \forall i\not\in I \}.
  \end{equation*}
  Note that $X^0$ does not contain any elements that dominate $x^0$,
  so $X^0_I$ is empty if $I$ is the empty set.  Also, no element of
  $X^0$ can be dominated by the minimal element $x^0$, so $X^0_I$ is
  empty if $I=\mathcal{I}$.

  Now, let $I$ be a subset of $\mathcal{I}$ such that $X^0_I$ is
  nonempty, and let $k=|I|$.  We have seen that necessarily $1\le k\le
  n-1$.  After reordering, we can write any $x\in X^0_I$ as
  $x=(x',x'')$ with the property that $x'\prec (x^0)'$ and $x''\succeq
  (x^0)''$, where $x'$ is made up of the first $k$ entries of $x$ and
  $x''$ of the remaining $n-k$ entries. To simplify the notation, consider the projections $p(x)=x'$ and $q(x)=x''$. 
Since $p(x)\prec (x^0)'$, there can only be finitely many such projections. So
let $\{z_i\}_{i=1}^{m_I}\subset\mathbb{N}_0^k$ be an enumeration of the
 nonnegative integer vectors satisfying $z_i=p(x)$ for some 
$x\in   X^0_I$ and define for each $i=1,2,\ldots,m_I$ the set
  \begin{equation*}
    X^0_{I,i} := \{ q(x)\in\mathbb{N}_0^{n-k} : (z_i,q(x))\in X^0_I \}.
  \end{equation*}
  Since $1\le n-k \le n-1$, the inductive hypothesis applies to
  $X^0_{I,i}$, and $X^0_{I,i}$ has a finite essential subset $Y^*$.
  Hence, the set $X^*_{I,i}$ of vectors $(z_i,y)$ with $y\in Y^*$ has the property that every $x\in
  X^0_I$ with $x=(z_i,x'')$ dominates some element of $X^*_{I,i}$.

  Define
  \begin{equation*}
    X^*_I = \bigcup_{i=1,2,\ldots,m_I} X^*_{I,i},\qquad\text{and}\qquad
    X^* = \{x_0\}\cup\bigcup_{I\in 2^{\mathcal{I}}} X^*_I
  \end{equation*}
  (using the convention that $X^*_I = \varnothing$ if $X^0_I =
  \varnothing$, and ``undoing'' the reorderings used to simplify the
  construction of $X^*_I$).  As the finite union of finite sets, $X^*$
  is finite.  Moreover, given any $x\in X$, either $x_0\preceq x$ or
  $x\in X^0_{I,i}$ for some $I$ and $i$ and thus dominates some
  $x^*\in X^*_{I,i}$.
\end{proof}

\begin{proof}[Proof of Theorem~\ref{thm:essential-subfamily}]
  The mapping of  $\Gamma$ into $\mathbb{N}_0^{|E|}$ defined 
  in~\eqref{eq:curves-to-N} preserves partial orders,  as described in (\ref{eq:orderpreserve}).
  Let $X = \{ x_\gamma : \gamma\in\Gamma \}$ and let $X^*\subseteq X$
  be the finite essential subset guaranteed by
  Theorem~\ref{thm:essential-subset}.  We may construct $\Gamma^*$ by
  choosing, for each $x^*\in X^*$, a corresponding $\gamma\in\Gamma$
  such that $x_\gamma = x^*$.  Then $\Gamma^*$ is finite, and every
  $\gamma\in\Gamma$ dominates some element of $\Gamma^*$.  Since
  $\Gamma^*\subseteq \Gamma$, we have that $A(\Gamma^*)\supseteq
  A(\Gamma)$.  Moreover, if $\rho\in A(\Gamma^*)$ and
  $\gamma\in\Gamma$ then there must exist a $\gamma^*\in\Gamma^*$ so
  that
  \begin{equation*}
    \ell_\rho(\gamma) \ge \ell_\rho(\gamma^*) \ge 1.
  \end{equation*}
  Thus, we also have the reverse inequality $A(\Gamma^*)\subseteq
  A(\Gamma)$.
\end{proof}

\section{Modulus as a Convex Program and the Karush-Kuhn-Tucker conditions}
\label{sec:kkt}

As mentioned above, upon enumerating the edges of the graph, each density $\rho$ can be thought as a vector in $\R^{|E|}$: $\rho=(\rho(e_1), \dots,\rho(e_n))$.
Also, by (\ref{eq:curves-to-N}), every walk $\ga$ corresponds to a vector $x_\ga\in \N_0^{|E|}\subset\R^{|E|}$. Moreover, by (\ref{eq:ellrhom}), the $\rho$-length of $\ga$ is simply the dot product in $\R^{|E|}$ of the vectors $\rho$ and $x_\ga$. In particular, given a walk $\ga$, the set $\{\rho\in \R^{|E|}: \ell_\rho(\ga)\geq 1\}$ is a closed half-space.
And given a family of walks $\Ga$, the admissible set $A(\Ga)$ is an intersection of closed half-spaces in $\R^{|E|}$, hence is convex. 

Therefore, since the energy $\cE(\rho)$ is a convex function of $\rho$, the problem of
computing the modulus of $\Gamma$ can be categorized as a ``standard
convex optimization problem'', namely one of minimizing a convex functional on a convex set.
\begin{equation}
  \label{eq:cvx-opt-orig}
  \begin{split}
    \text{minimize}\qquad &\cE(\rho)\\
    \text{subject to}\qquad &\rho\in A(\Gamma)
  \end{split}
\end{equation}
We see that existence and uniqueness of a minimizer for this problem  follows by compactness and strict convexity of the Euclidean balls in $\R^{|E|}$. This proves Theorem \ref{thm:existuniq}.
In fact, the same applies to the generalized energy $\cE_p(\rho)\defeq \sum_i \rho(e_i)^p$, for any
$p> 1$.

Furthermore,
 Theorem~\ref{thm:essential-subfamily} implies that $\Gamma$
can be taken to be finite.  Thus, the constraint $\rho\in A(\Gamma)$
can be replaced by the finite set of inequalities $\rho \ge 0$ and
$\ell_\rho(\gamma_i)\ge 1$ for $i=1,2,\ldots,m$, where
$\{\gamma_i\}_{i=1}^m=\Gamma$.  Next, we note that the requirement
$\rho\ge 0$ can be omitted altogether.  Define
\begin{equation*}
  \tilde{A}(\Gamma) := \{ \rho:E\to\mathbb{R}\;:\;
  \ell_\rho(\gamma)\ge 1\;\forall\gamma\in\Gamma \}.
\end{equation*}
Neither the minimum value in \eqref{eq:cvx-opt-orig}, nor the extremal
$\rho$ will change if the constraint $\rho\in A(\Gamma)$ is replaced
by $\rho\in\tilde{A}(\Gamma)$; the property $\rho\ge 0$ will be
automatically satisfied by the extremal $\rho$.  Indeed, given any
$\rho\in\tilde{A}(\Gamma)$, the positive part $\rho_+(e) :=
\max\{\rho(e),0\}$ is in $A(\Gamma)$ and satisfies
$\cE(\rho_+)\le\cE(\rho)$.  Thus, the constraint $\rho\ge0$ need not
be explicitly enforced.

Replacing the constraint by the equivalent finite set of inequalities,
problem~\eqref{eq:cvx-opt-orig} takes the form of an ``ordinary convex
program''~\cite[Sec.~28]{Rockafellar}
\begin{equation}
  \label{eq:cvx-opt}
  \begin{split}
    \text{minimize}\qquad &\cE(\rho)\\
    \text{subject to}\qquad & 1 - \ell_\rho(\gamma_i) \le 0\qquad i=1,2,\ldots,m.
  \end{split}
\end{equation}
where the minimum is taken over all real valued functions
$\rho:E\to\mathbb{R}$.  This is an ordinary convex program with affine
inequality constraints.

Define the Lagrangian
\begin{equation}
  \label{eq:cvx-lagrangian}
  L(\rho,\lambda) := \cE(\rho) + \sum_{i=1}^m\lambda_i\left(1 - \ell_\rho(\gamma_i)\right)
\end{equation}
associated with problem~\eqref{eq:cvx-opt}.  We recall that the
optimization problem is said to have the property of \emph{strong
  duality} if the Lagrangian has a saddle point.  That is,
problem~\eqref{eq:cvx-opt} exhibits strong duality if there exist
$\rho^*:E\to\mathbb{R}$ and $\lambda^*\in\mathbb{R}^m_+$ (the set of
$m$-dimensional vectors with non-negative entries) such that
\begin{equation}\label{eq:saddle}
  L(\rho^*,\lambda) \le L(\rho^*,\lambda^*) \le L(\rho,\lambda^*)
  \qquad\forall \rho:E\to \mathbb{R},\quad \lambda\in\mathbb{R}^m_+\,.
\end{equation}

For convex problems, Slater's condition (see
also~\cite[Theorem 28.2]{Rockafellar}) ensures strong duality if a
\emph{strictly feasible} point exists.  In our context, a choice of
$\rho$ is called \emph{strictly feasible} if all inequality
constraints are satisfied strictly: $\ell_\rho(\gamma_i)>1$ for
$i=1,2,\ldots,m$.  By choosing $\rho(e)=2$ for every $e\in E$, it is clear that strictly feasible choices for $\rho$
exist.

Since the problem is convex, sufficiently smooth, and exhibits strong
duality, the Karush-Kuhn-Tucker (KKT) conditions provide necessary and
sufficient conditions for optimality.  Stated for the present problem,
the KKT conditions ensure the existence  of  an optimal
$\rho^*:E\to\mathbb{R}$ and dual optimal
$\lambda^*\in\mathbb{R}^m_+$~\cite[Theorem 28.3]{Rockafellar} satisfying
\begin{gather}
  \lambda^*_i \ge 0,\quad 1 - \ell_{\rho^*}(\gamma_i) \le 0
  \quad  \text{for }i=1,2,\ldots,m\\
  \lambda^*_i\left(1 - \ell_{\rho^*}(\gamma_i)\right) = 0\quad
  \text{for }i=1,2,\ldots,m\label{eq:comp-slack}\\
  \nabla_\rho L(\rho^*,\lambda^*) = 0
\end{gather}
The density $\rho^*$ is the minimizer of~\eqref{eq:cvx-opt}.

The values of the Lagrange multipliers $\lambda^*$ provide a
straightforward interpretation of $\tilde{\Gamma}$ in Beurling's
Criterion.
\begin{theorem}
Assume $\Gamma$ is a finite (and enumerated) family of walks.
Let $\rho^*$ and $\lambda^*$ be a saddle point for the Lagrangian
  $L$ as in {\rm (\ref{eq:saddle})}.  Then
  \begin{equation*}
    \tilde{\Gamma} \defeq \left\{\gamma_i: \lambda_i^*>0\right\}
  \end{equation*}
  can be taken to be the subfamily 
  $\tilde{\Gamma}$ of $\Gamma$ in  Beurling's Criterion
  (Theorem~\ref{thm:beurling}).
\end{theorem}
\begin{proof}
  First, note that the ``complementary slackness''
  condition~\eqref{eq:comp-slack} implies that
  $\tilde{\Gamma}\subset \Gamma_{0}(\rho^*)$ as needed by the
  criterion.  Now, suppose $h:E\to\mathbb{R}$ satisfies
  $\ell_h(\ga) \ge 0$, for all  $\gamma\in\tilde{\Gamma}$,
  and that $\mu>0$.  The saddle-point property requires that
  \begin{equation*}
    \begin{split}
      0 &\le L(\rho^*+\mu h,\lambda^*) - L(\rho^*,\lambda^*) \\
      &= \cE(\rho^*+\mu h) + \sum_{\substack{i=1\\\lambda^*_i>0}}^m\lambda^*_i\left(1-\ell_{\rho^*+\mu h}(\gamma_i)\right)
       - \cE(\rho^*)\\
       &= 2\mu\sum_{e\in E}\rho^*(e)h(e) + \mu^2\cE(h) - \mu\sum_{\substack{i=1\\\lambda^*_i>0}}^m\lambda^*_i\ell_h(\gamma_i),
    \end{split}
  \end{equation*}
  which holds if and only if
  \begin{equation*}
    2\sum_{e\in E}\rho^*(e)h(e) \ge \sum_{\substack{i=1\\\lambda^*_i>0}}^m\lambda^*_i\ell_h(\gamma_i) - \mu\cE(h).
  \end{equation*}
  Since $\mu>0$ is arbitrary and since, by hypothesis,
  $\ell_h(\gamma_i)\ge 0$ whenever $\lambda^*_i>0$, this implies
  \begin{equation*}
    \sum_{e\in E}\rho^*(e)h(e) \ge 0.
  \end{equation*}
\end{proof}

\section{An algorithm for approximating the modulus}

\subsection{Algorithm}
Although~\eqref{eq:cvx-opt} provides a convex programming
characterization of the modulus problem, it may not be particularly
useful in practice.  The problem is that, although
Theorem~\ref{thm:essential-subfamily} shows that $\Gamma$ can be taken
to be finite, the number of walks in $\Gamma$ may still be very large.
Consider, for example, the graph shown in
Figure~\ref{fig:choked-graph}, where $\Gamma$ is the set of all walks
connecting the red and blue nodes.  More generally, consider a graph
with $N$ nodes with the property that the first $N-1$ nodes induce a
complete subgraph and the $N$th node has node $1$ as its sole
neighbor.  Let $\Gamma$ be the set of all walks connecting node $2$ to
node $N$.  It is straightforward to verify that the set
$\Gamma^*\subset\Gamma$ defined as the set of all simple paths
connecting $2$ to $N$ form an essential subfamily of $\Gamma$, and
every path in $\Gamma^*$ is minimal in the sense of (\ref{eq:curve-partial-ordering}).  However, only counting the
paths in $\Gamma^*$ that traverse all nodes, we already have
$|\Gamma^*| \ge (N-3)!$, which corresponds to an enormous number of constraints, even
for moderately large graphs.  On the other hand, since the edge
connecting nodes $1$ and $N$ provides a ``choke point'' through which
all curves in $\Gamma$ must pass, it is natural to expect that there
may exist a relatively small subset $\Gamma'\subset\Gamma$ such that
$\Mod(\Gamma')\approx\Mod(\Gamma)$.

\begin{figure}
  \centering
  \includegraphics[width=0.75\textwidth]{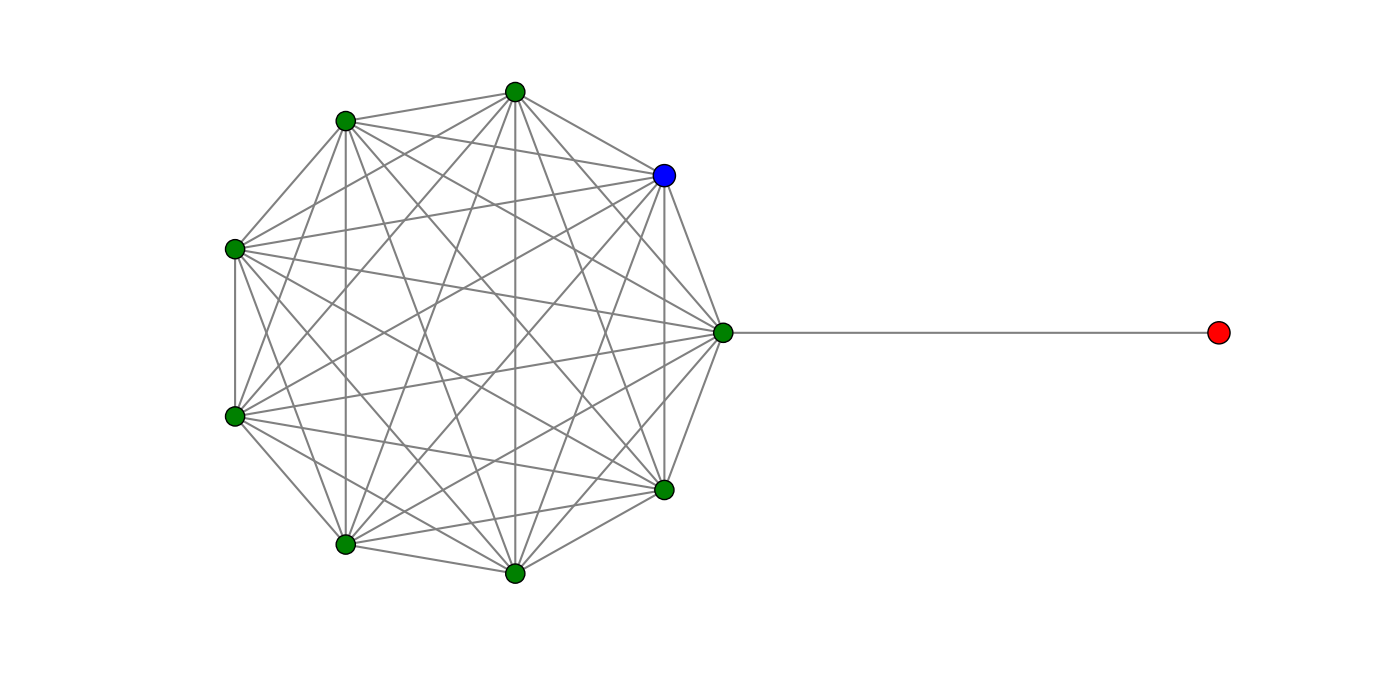}
  \caption{Only a relatively small number of paths are required to
    approximate the modulus of the set of all walks connecting the red
    and blue nodes.}
  \label{fig:choked-graph}
\end{figure}

In this section, we present an algorithm for approximating the modulus
of a family of curves $\Gamma$ which performs well in cases where such
a $\Gamma'$ exists.  In the following, we assume that there exists an
algorithm, denoted $\texttt{shortest}(\rho)$, which produces a
$\rho$-shortest walk in $\Gamma$.  That is, given
$\rho:E\mapsto[0,\infty)$
\begin{equation*}
  \gamma^* = \texttt{shortest}(\rho) \qquad\implies\qquad
  \forall\gamma\in\Gamma:\, \ell_\rho(\gamma^*)\le\ell_\rho(\gamma).
\end{equation*}
In the example above, for instance, $\texttt{shortest}$ can be
implemented by means of Dijkstra's algorithm, see \cite[Theorem 7.15]{aigner}.  Pseudocode for the
approximation algorithm is given in Algorithm~\ref{alg:mod}.
Since our arguments do not require substantial modifications, we consider the more general case of $p$-modulus $\Mod_p(\Gamma) \defeq \inf_{\rho \in A(\Gamma )} \cE_p(\rho)$, for $1< p<\infty$.

\begin{algorithm}
  \begin{algorithmic}
    \STATE $\rho \leftarrow 0$ \STATE $\Gamma' \leftarrow
    \varnothing$ \LOOP \STATE $\gamma \leftarrow
    \texttt{shortest}(\rho)$ \IF{$\ell_\rho(\gamma)^{p} \ge
      1-\etol$} \STATE \texttt{stop}
    \ENDIF
    \STATE
    $\Gamma'\leftarrow\Gamma'\cup\{\gamma\}$
    \STATE $\rho\leftarrow\argmin\{\cE_p(\rho) : \rho\in
    A(\Gamma')\}$
    \ENDLOOP
  \end{algorithmic}
  \caption{Approximates $\Mod_p(\Gamma)$ with an error tolerance of
    $0<\etol<1$.}
  \label{alg:mod}
\end{algorithm}

During each iteration through the loop, a shortest walk is chosen from
$\Gamma$ using the provided \texttt{shortest} algorithm.  If the
stopping criterion is not met, the new walk is added to $\Gamma'$ and
the convex optimization problem described in Section~\ref{sec:kkt} is
solved using a standard convex programming code.

\begin{theorem} Let $\Gamma$ be a family of walks on a finite graph and suppose that $\rho^\star$ is the extremal density for $\Mod_p(\Gamma)$ with $1<p<\infty$. Fix an error tolerance $0<\etol<1$.  Then,  Algorithm~\ref{alg:mod} will terminate in finite time, and will output a subfamily $\Gamma'\subset\Gamma$ and a density $\rho$. 

Moreover, $\Gamma'$ has the property that
  \begin{equation}
0\leq   \frac{\Mod_p(\Gamma)- \Mod_p(\Gamma')}{\Mod_p(\Ga)} \le \etol,
    \label{eq:alg-mod}
  \end{equation}
while $\rho$ satisfies
{\renewcommand\arraystretch{2}
\begin{equation}\label{eq:alg-rho}
\frac{\|\rho^\star-\rho\|_p}{\|\rho^\star\|_p}\leq \left\{\begin{array}{ll}2^{1-1/p} \left(p\etol\right)^{1/p} &\qquad \mbox{if $p\geq 2$}\\
2^{1/p}\left(\frac{p\etol}{p-1}\right)^{1-1/p}&\qquad\mbox{if $1<p<2$}
\end{array}\right.
\end{equation}}
\end{theorem}
\begin{remark}\label{rem:pone}
The inequality (\ref{eq:alg-mod}) holds even when $p=1$. However, (\ref{eq:alg-rho}) depends on the strict convexity of the $p$-norm when $p>1$. Also the lower bound in \eqref{eq:alg-mod} follows from the monotonicity of modulus.
\end{remark}
\begin{proof}
  Since the size of $\Gamma'$ is monotonically increasing within the
  loop and is bounded from above by $|\Gamma|$, which by Theorem \ref{thm:essential-subfamily} we can assume to be finite, the algorithm will certainly terminate in
  finite time.  (As demonstrated in the examples to follow, however,
  for certain graphs the algorithm tends to terminate after relatively
  few iterations.)  

  To see~\eqref{eq:alg-mod}, we observe that
 the identity  $\cE_p(\rho)=\Mod_p(\Gamma')$ is a loop invariant.  Moreover, the loop
  can only terminate if $\ell := \ell_\rho(\gamma) > 0$.  But when $\ell >
  0$, then the definition of $\gamma$ implies that $\frac{1}{\ell}\rho\in
  A(\Gamma)$ and, thus, that
  \begin{equation*}
    \Mod_p(\Gamma) \le \cE_p\left(\frac{\rho}{\ell}\right)
    = \frac{1}{\ell^p}\cE_p(\rho) = \frac{1}{\ell^p}\Mod_p(\Gamma').
  \end{equation*}
  When the loop terminates, $\ell^p \ge 1-\etol$, which implies 
\[
\Mod_p(\Ga)\leq (1-\etol)^{-1}\Mod_p(\Ga').
\]
 This provides
  an estimate of the relative error upon termination:
  \begin{equation*}
    \frac{\Mod_p(\Gamma)-\Mod_p(\Gamma')}{\Mod_p(\Gamma)}=1-\frac{\Mod_p(\Ga')}{\Mod_p(\Ga)} \le 1-(1-\etol)=\etol.
  \end{equation*}

In order to prove (\ref{eq:alg-rho}), consider the Clarkson inequality for $p\geq 2$:
\begin{equation}\label{eq:parallel}
\cE_p(\rho+\rho^\star)+\cE_p(\rho-\rho^\star)\leq 2^{p-1}\left(\cE_p(\rho)+\cE_p(\rho^\star)\right).
\end{equation}
Notice that $(\rho+\rho^\star)/(1+(1-\etol)^{1/p})\in A(\Gamma)$. So
\begin{equation}\label{eq:sum}
\cE_p(\rho+\rho^\star)\geq (1+(1-\etol)^{1/p})^p\Mod_p(\Gamma).
\end{equation}
Also $\cE_p(\rho^\star)=\Mod_p(\Gamma)$ and $\cE_p(\rho)=\Mod_p(\Gamma')\leq\Mod_p(\Gamma)$, by (\ref{eq:alg-mod}). Therefore, by (\ref{eq:sum}), inequality (\ref{eq:parallel}) becomes
\begin{align*}
\cE_p(\rho-\rho^\star)  &\leq 2^p\left(1-\left(\frac{1+(1-\etol)^{1/p}}{2}\right)^p\right)\Mod_p(\Gamma)\\
&\leq 2^{p-1}p\left(1-(1-\etol)^{1/p}\right)\Mod_p(\Gamma)&\text{(since $(1+t)^p\ge 1+pt$)} \\
& \leq 2^{p-1}p\Mod_p(\Gamma)\etol& \text{(since $0<\etol<1$).} 
\end{align*}

Rewriting we obtain that
\[
\|\rho^\star-\rho\|_p^p\leq 2^{p-1}p\|\rho^\star\|_p^p\etol.
\]

When $1<p<2$, the Clarkson inequality reads:
\begin{equation}\label{eq:parallel2}
\|\rho+\rho^\star\|_p^q+\|\rho-\rho^\star\|_p^q\leq 2\left(\|\rho\|_p^p+\|\rho^\star\|_p^p\right)^{q/p}
\end{equation}
where $q=p/(p-1)$.

Therefore, as before by (\ref{eq:sum}) and the first inequality in (\ref{eq:alg-mod}),  (\ref{eq:parallel2}) becomes
\begin{align*}
\|\rho-\rho^\star\|_p^q  &\leq 2\left(\|\rho\|_p^p+\|\rho^\star\|_p^p\right)^{q/p}-\left(1+(1-\etol)^{1/p}\right)^q\|\rho^\star\|_p^q\\
&\leq \left[2^{1+q/p}-\left(1+(1-\etol)^{1/p}\right)^q\right]\|\rho^\star\|_p^q\\
&\leq  2^q\left[1-\left(\frac{1+(1-\etol)^{1/p}}{2}\right)^q\right]\|\rho^\star\|_p^q\\
&=2^{q-1}q\etol \|\rho^\star\|_p^q.
\end{align*}
\end{proof}

\subsection{Examples}

\subsubsection{Choked graph}
In the case of the ``choked graph'' described at the beginning of this
section (see also Figure~\ref{fig:choked-graph}), the algorithm
$\texttt{shortest}(\rho)$ can be taken as an implementation of
Dijkstra's algorithm.  As described previously, the obvious essential
subfamily $\Gamma^*\subset\Gamma$ of simple paths in $\Gamma$ is made
up of minimal elements (in the sense of the
ordering~\eqref{eq:curve-partial-ordering}) and has at least
$|\Gamma^*|\ge (N-3)!$ elements.  However, due to the ``choke point''
in the graph, it seems reasonable to expect that $\Mod(\Gamma)$ can be
approximated by many fewer paths than this.
Table~\ref{tab:choked-graph} demonstrates the results of
Algorithm~\ref{alg:mod} with tolerance $\etol=10^{-2}$ applied to the
$N$-node choked graph for several values of $N$.  The ``true'' values
of $\Mod(\Gamma)$ in the table were computed via the effective
resistance matrix, see Section \ref{ssec:effcond}.  In all cases, only a modest $\mathcal{O}(N)$
number of paths $\Gamma'\subset\Gamma$ are required to approximate
$\Mod(\Gamma)$ to two decimal places.
\begin{table}
  \centering
  \begin{tabular}{c|c|c|c|c}
    $N$ & $(N-3)!$ & $|\Gamma'|$ & $\Mod(\Gamma')$ &  $\Mod(\Gamma)$ \\
    \hline
    10 & $5040$ & 8 & 0.81818182 & 0.81818182\\
    40 & $1.38\times 10^{43}$ & 38 & 0.95121951 & 0.95121951\\
    160 & $1.17\times 10^{278}$ & 158 &0.98757764 & 0.98757764\\
    640 & $2.47\times 10^{1511}$ & 400 & 0.99503722 & 0.99687988
    \end{tabular}
    \vspace{2em}
    \caption{The results of applying Algorithm~\ref{alg:mod} to the 
      choked graph (see Figure~\ref{fig:choked-graph}) with $N$ nodes.
      As described in the text, the
      obvious choice of essential subfamily has $|\Gamma^*|\ge(N-3)!$
      paths.  However, the algorithm seems to require less than $N$
      iterations to resolve $\Mod(\Gamma)$ to within a relative error
      of $10^{-2}$ or less, as the column $|\Gamma^\prime|$ indicates.}
  \label{tab:choked-graph}
\end{table}

\subsubsection{Random $G(n,p)$ graphs}

Intuitively, one might expect that for sparse graphs there will be a
relatively small number of choke points similar to the one  in the
choked graph example; choosing a few key walks accessing those choke
points would hopefully be sufficient to approximate $\Mod(\Gamma)$.
We can obtain an empirical understanding of the behavior of the
algorithm as follows.  For fixed $n$, we take $p$ to be a random real
number chosen uniformly from the range $2\log(n)/n \le p \le n$ and
select a random graph $G$, chosen according to the Erd\H{o}s-R\'{e}nyi
$G(n,p)$ model: that is, $G$ is a graph with $n$ vertices, with each possible
edge selected independently for inclusion in the edge set with
probability $p$.  By choosing $p\ge 2\log(n)/n$, we ensure a high
likelihood that $G$ will be connected \cite[Theorem 2.8.1]{durrett}.  We then apply
Algorithm~\ref{alg:mod} with tolerance $\etol=10^{-2}$ to the set
$\Gamma$ of all walks connecting the first two vertices in
$G$.  (Again, $\texttt{shortest}$ is based on Dijkstra's algorithm.)

A natural guess for $\Gamma^*$ is the set $\Gamma_s$ of all simple
paths connecting the two vertices; every walk in $\Gamma$ contains a
simple subpath, and each simple path is minimal.  However, as will be
demonstrated, this choice of $\Gamma^*$ is much larger
than the set $\Gamma'$ needed to approximate the modulus.  We begin by
estimating $\mathbb{E}(|\Gamma_s|)$, the expected number of simple
paths connecting the first two nodes in a $G(n,p)$ graph.  For each
simple path $\gamma$ in the complete $n$-node graph, define
$\mathbf{1}_\gamma$ as the random variable which takes the value $1$
for any $G(n,p)$ graph that contains $\gamma$, and $0$ for all others.
Then
\begin{equation*}
  \mathbb{E}(|\Gamma_s|) = \mathbb{E}\left(\sum_\gamma\mathbf{1}_\gamma\right)
  = \sum_\gamma\mathbb{E}\left(\mathbf{1}_\gamma\right).
\end{equation*}
Since all edges are selected independently with probability p, the
expected value of $\mathbf{1}_\gamma$ for any $k$-hop simple path
$\gamma$ is $p^k$.  Each such path passes through $k-1$ vertices in
addition to the specially selected first and second, so there are
$(n-2)!/(n-k-1)!$ possible paths with $k$ hops.  Thus, we have
\begin{equation*}
  \mathbb{E}(|\Gamma_s|) = \sum_{k=1}^{n-1}\frac{(n-2)!}{(n-k-1)!}\,p^k.
\end{equation*}
The magnitude of this number can be approximated via Stirling's
approximation:
\begin{equation*}
  \log n! \ge
  \left(n+\frac{1}{2}\right)\log n - n + \frac{1}{2}\log 2\pi\,.
\end{equation*}
Approximating the sum by the final term, we arrive at the inequality
\begin{equation}
  \label{eq:approx-gamma-s}
  \begin{split}
    \log \mathbb{E}(|\Gamma_s|) &\ge (n-1)\log p +
    \left(n-\frac{3}{2}\right)\log (n-2) - (n-2) + \frac{1}{2}\log
    2\pi.
  \end{split}
\end{equation}

\begin{figure}[H]
  \centering
  \subfigure{
    \includegraphics[width=0.47\textwidth]{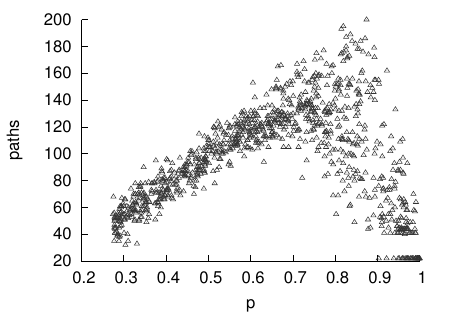}
  }
  \subfigure{
    \includegraphics[width=0.47\textwidth]{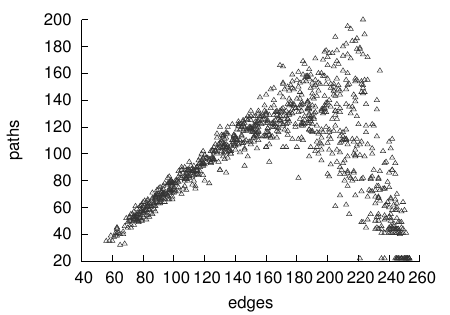}
  }
  \caption{Number of walks $|\Gamma'|$ required to approximate the
    modulus of walks between two points in a random sampling of
    $G(23,p)$ graphs plotted against the parameter $p$ (left) and
    versus the size $|E|$ (right).}
  \label{fig:random-gnp-23}

  \centering
  \subfigure{
    \includegraphics[width=0.47\textwidth]{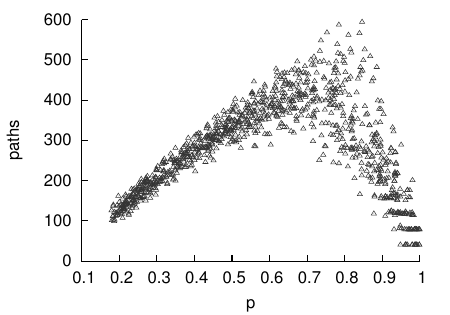}
  }
  \subfigure{
    \includegraphics[width=0.47\textwidth]{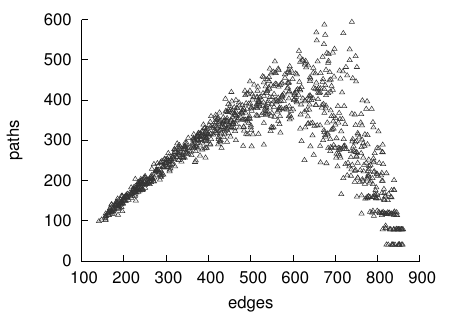}
  }
  \caption{Number of walks $|\Gamma'|$ required to approximate the
    modulus of walks between two points in a random sampling of
    $G(42,p)$ graphs plotted against the parameter $p$ (left) and
    versus the size $|E|$ (right).}
  \label{fig:random-gnp-42}
\end{figure}

The plots in Figure~\ref{fig:random-gnp-23} show the number of curves
in $\Gamma'$ upon algorithm termination versus the parameter $p$ and
number of $|E|$ for random $G(n,p)$ graphs with $n=23$ vertices
obtained as described above from $1000$ random graph samples.  In all
$1000$ graphs, $|\Gamma'|=200$ was sufficient to compute the modulus
to within the $\etol=10^{-2}$ tolerance.  For comparison,
Equation~\eqref{eq:approx-gamma-s} approximates the size of $\Gamma_s$
to be larger than $10^{13}$ for $p\ge 0.5$.

Figure~\ref{fig:random-gnp-42} displays similar plots produced from
$1000$ random graphs with $n=42$ vertices.  In this case,
$|\Gamma'|=600$ was sufficient to compute the modulus to within the
$\etol=10^{-2}$ tolerance.  For comparison,
Equation~\eqref{eq:approx-gamma-s} approximates the size of $\Gamma_s$
to be larger than $10^{47}$ for $p\ge 0.5$.

It would be interesting to be able to give a theoretical explanation for the inverted-parabola nature of the graphs in Figure~\ref{fig:random-gnp-23} and Figure~\ref{fig:random-gnp-42}.

\subsubsection{Modulus of ``via'' walks}

\begin{figure}[H]
  \centering
  \subfigure{
    \includegraphics[width=0.47\textwidth]{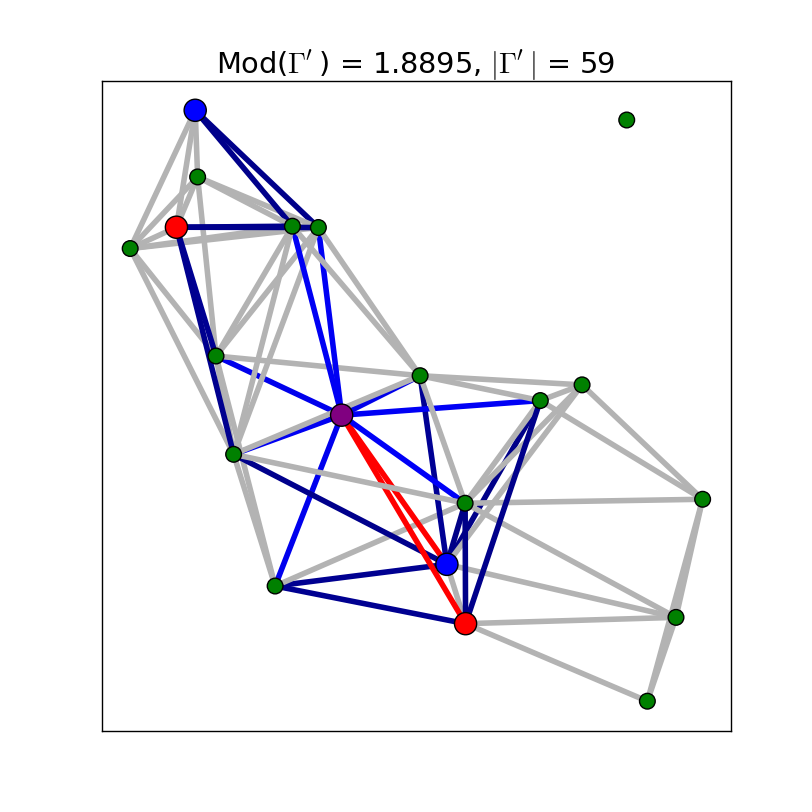}
  }
  \subfigure{
    \includegraphics[width=0.47\textwidth]{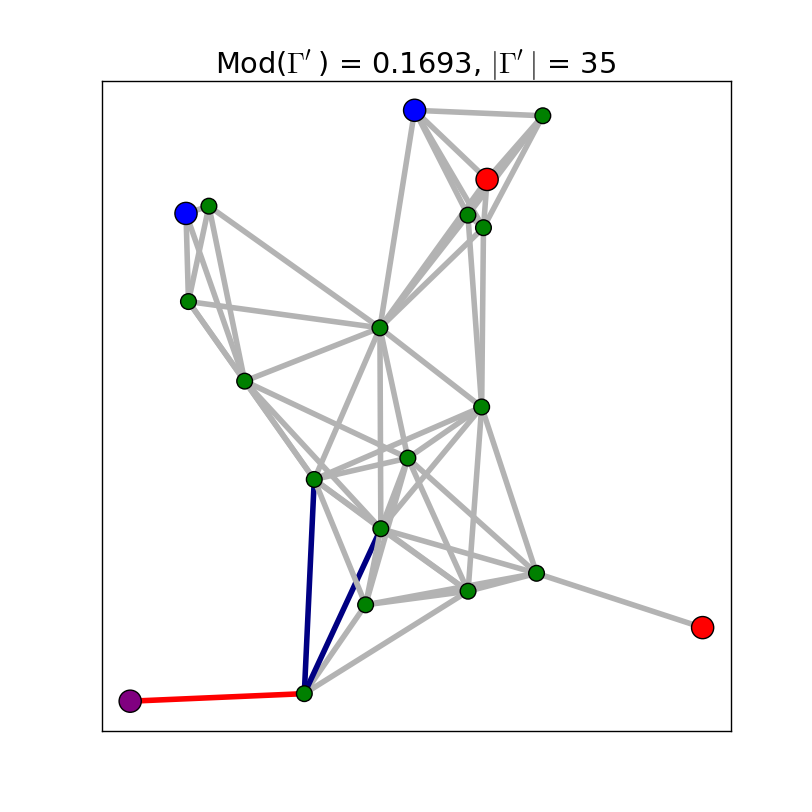}
  }
  \caption{Examples of computing the modulus of the set of all walks
    beginning at one of the large red nodes, passing through the large
    purple node, and ending at one of the large blue nodes.  The
    approximate modulus and size $|\Gamma'|$ (for $\etol=10^{-2}$) are
    reported for both examples.}
  \label{fig:thru-graphs}
\end{figure}

In order to demonstrate the flexibility of the proposed algorithm, we
consider a final example: the modulus of walks connecting two disjoint
sets of nodes $A$ and $B$ via a particular node $c$.  That is,
$\Gamma$ consists of all walks beginning in $A$ and ending in $B$
which pass through the node $c$ along the way.  In this case,
\texttt{shortest} can be implemented efficiently by iteration of
Dijkstra's algorithm.  First, the algorithm determines the shortest
walk from $A$ to $c$, then from $c$ to $B$.  The concatenation of
these two walks gives the shortest walk from $A$ to $B$ via $c$.
Figure~\ref{fig:thru-graphs} provides a visualization of the results
in two different configurations.  In the figures, the edge with
highest weight is drawn in bright red.  All edges whose weights are at
least $75\%$ of the highest weight are drawn in red, with brighter
hues representing higher weights.  All edges weighted between $25\%$
and $75\%$ of the highest weight are drawn in blue, with brighter hues
representing higher weights.  The remaining edges are drawn in gray.

\section{Conclusion}

We established the basic theory of families of walks in graphs and their modulus. This is a concept that generalizes the fundamental notion of effective conductance in electrical networks. The theory is derived from its continuous analog, where modulus has had many applications. The advantage of finite graphs is their potential for computations. We propose to use walks instead of curves, because they're more suitable to numerical algorithms, especially ones that try to find shortest walks within a family. We develop an algorithm for computing the modulus of families of walks on graphs. Our experiments show this algorithm to be fairly fast. However, more experiments and more theory is needed to obtain a complete comparison of our proposed algorithm with other possible ones. See \cite{ABPPW2015} for more theoretical investigations of modulus, and \cite{SPSA,GASSP} for examples of applications.

\bibliographystyle{amsplain}
\bibliography{references}

\end{document}